\documentclass[12 pt]{amsart}
\usepackage{amsfonts,amsmath,amsthm,amssymb, amscd}
\usepackage[top=2cm, bottom=5cm, left=3cm, right=1cm]{geometry}

\usepackage{tikz-cd}
\usepackage{xcolor}
\usepackage{placeins}
\usepackage{hyperref}

\newtheorem{theorem}{Theorem}[section]
\theoremstyle{definition}

\newtheorem{corollary}[theorem]{Corollary}
\newtheorem{lemma}[theorem]{Lemma}
\newtheorem{proposition}[theorem]{Proposition}
\newtheorem{question}[theorem]{Question}
\newtheorem{problem}[theorem]{Problem}
\newtheorem{definition}[theorem]{Definition}
\newtheorem{example}[theorem]{Example}
\newtheorem{remark}[theorem]{Remark}

\newcommand{\Alex}{\operatorname{Alex}}
\newcommand{\Aut}{\operatorname{Aut}}

\newcommand{\Conj}{\operatorname{Conj\,}}

\newcommand{\Core}{\operatorname{Core}}
\newcommand{\im}{\operatorname{Im}}

\newcommand{\Inn}{\operatorname{Inn}}

\newcommand{\T}{\operatorname{T}}

\newcommand{\R}{\operatorname{R}}

\newcommand{\id}{\mathrm{id}}

\newcommand{\Hol}{\operatorname{Hol}}

\usepackage{lineno}
\usepackage{color}

\title[Rota--Baxter and averaging  operators]
{
Rota--Baxter and averaging  operators on racks and rack algebras}
\author[Bardakov, Bovdi]{V.~G.~Bardakov, V.~A.~Bovdi}

\address{Sobolev Institute of Mathematics, 4 Acad. Koptyug avenue, 630090, Novosibirsk, Russia.}
\address{Novosibirsk State Agrarian University, Dobrolyubova street, 160, Novosibirsk, 630039, Russia.}
\address{Regional Scientific and Educational Mathematical Center of Tomsk State University, 36 Lenin Ave., Tomsk, Russia.}
\email{bardakov@math.nsc.ru}

\address{United Arab Emirates University, P.O. Box 15551, Al Ain, Abu Dhabi, United Arab Emirates.}
\email{vbovdi@gmail.com}
\date{\today}

\begin{document}
\maketitle
\begin{abstract}
In the present article we define and investigate relative Rota--Baxter operators and  relative averaging operators   on racks  and rack  algebras. 
Also, if  $B$ is a Rota--Baxter or  averaging operator on a rack $X$, then we can extend $B$ by linearity to the rack algebra $\Bbbk[X]$. On the other side, we have definitions of Rota--Baxter and averaging operators  on arbitrary  algebra. We find connections between these operators. In particular, we prove that if $B : X \to X$ is an averaging operator on a rack, then its linear extension on a rack algebra $\Bbbk[X]$ gives an averaging operator.


\end{abstract}
\maketitle
\tableofcontents

\section{Introduction}

The study of operators on different  algebraic systems is a classic direction in algebra. The most studied operators over these objects are their endomorphisms and automorphisms. To define an endomorphism of an algebraic system, it is sufficient to define it on the generating set, and then we extend it to all elements.
However, there are other types of operators, the construction of which is not so simple. Examples of such operators include the Reynolds operator and the averaging operator, which arose independently in different areas of mathematics. Let $X$ be a locally compact Hausdorff space and let $C_0(X)$ denotes the Banach algebra of all continuous functions on $X$ vanishing at $\infty$. A Reynolds operator $T:  C_0(X) \to C_0(X)$ (see \cite{R}) is a linear operator, satisfying the Reynolds identity:
\[
T(a) T(b) = T \left( T(a) b + a T(b) - T(a) T(b)\right), \qquad (a, b \in C_0(X)).
\]
Obviously,  a projection $T$ is a Reynolds operator if  $T$ is averaging, i.e.
\[
T(a) T(b) = T \left( a T(b) \right), \qquad (a, b \in C_0(X)).
\]
Reynolds introduced this operator in connection with the theory of turbulence. Averaging operators are studied in the works of Birkoff \cite{Bir1, Bir2} and further by Kelley \cite{K}, Rota \cite{R1, R2}, Seever \cite{S} and many others (see \cite{HP, FM, KKL, D, PB, DM, BCEM}, as well as references therein).

One type of environment of many interesting operators is the Rota-Baxter operator. The first appearance of Rota--Baxter operators (shortly, RB-operators) for commutative algebras can be traced back to Baxter's seminal paper \cite{Bax} in 1960. Since then, the theory of Rota-Baxter operators has undergone extensive development by various authors in different fields of mathematics. It should be noted that Rota-Baxter operators play an important role due to their connection with a number of mathematical concepts, including the Yang-Baxter equation \cite{MR0674005, MR0725413}, Loday algebras \cite{MR3021790}, and double Poisson algebras \cite { MR2568415}.
A historical overview and connections of these operators with other mathematical concepts can be found in the monograph \cite{G}.

A  Rota--Baxter operator on groups were introduced in \cite{GLS}, a  group $G$ with a Rota--Baxter operator $B$ is called a Rota--Baxter group (shortly, RB-group). In \cite{GLS} it was shown that if $(G, B)$ is a Rota--Baxter Lie group, then the tangent map $B$ at identity is a Rota--Baxter operator of weight $1$ on the Lie algebra of the Lie group $G$. The properties of $RB$-groups are actively studied in \cite{BG}. In \cite{BG-1}, a connection between $RB$-groups,  the Yang--Baxter equation and skew braces  was found. In \cite{Gon} it was proven that the Rota--Baxter operators on the group ring $\Bbbk[G]$ are in one-to-one correspondence with the Rota--Baxter operators  of weight 1 on $G$.
In \cite{JSZ} relative Rota--Baxter groups were defined as groups with a relative Rota--Baxter operator. In this case, the relative Rota--Baxter operator depends not only on the group, but also on the  group of its automorphisms.  In the case when this group of automorphisms is a group of inner automorphisms, we obtain the Rota--Baxter operator.  Some properties of relative Rota--Baxter operators are studied in \cite{RS-23}.  Relative Rota--Baxter operators on an arbitrary Hopf algebra were defined in \cite{BN}.

Following Kurosh \cite{Kurosh} in the sequel  of this article we called a non-empty set with a  binary algebraic operation a groupoid. Our attention  concentrated on the investigation  of Rota--Baxter  and averaging operators  on such groupoids as groups, racks and quandles. Obviously, for any groupoid one can define averaging operator.
Recall that a quandle is a groupoid  with a binary operation that satisfies axioms encoding the three Reidemeister moves of planar diagrams of links in the $3$-space. If we forget on the first Reidemeister move we get a groupoid which is called by rack.
Racks and quandles were introduced by   Joyce \cite{Joyce} and Matveev \cite{Matveev}, who showed that link quandles are complete invariants of non-split links up to orientation of the ambient $3$-space.
It is worth emphasizing that racks and quandles are also useful in certain areas of mathematics: in group theory, set-theoretic solutions of the Yang-Baxter equations and Hopf algebras, and  discrete integrable systems (see \cite{AG}).
 Automorphisms of quandels, which reveal much about their internal structure, are studied in detail in a series of articles \cite{BDS, BarTimSin}.

 In the present article we define and study relative Rota--Baxter and  averaging operators  on  racks and racks algebras.
In particular, we are studying the following question.
Let $(X, *)$ be a rack (quandle) and $B \colon X \to X$ is a relative Rota--Baxter or  averaging operator on $X$.
Is it possible to define new algebraic binary operation $\circ$ on $X$ such that $(X, \circ)$ is a rack (quandle),  $B$ is a relative Rota--Baxter or  averaging operator on $(X, \circ)$ and $B$ induces a rack homomorphism $B : (X, \circ) \to (X, *)$?

Also, if  $B$ is a Rota--Baxter or  averaging operator on $X$, then we can extend $B$ by linearity to the rack algebra $\Bbbk[X]$. On the other side, we have definitions of Rota--Baxter and averaging operators  on arbitrary  algebra. The question on connection  between these two definitions arises.
We prove that if $B : X \to X$ is an averaging operator on a rack, then its linear extension on a rack algebra $\Bbbk[X]$ gives an averaging operator. The situation with Rota--Baxter operators  is more complicated. We give conditions under which the linear extension of a Rota--Baxter operator $B : X \to X$ on $\Bbbk[X]$ (we shall call this extension by rack Rota--Baxter operator) gives an algebraic Rota--Baxter operator on the algebra $\Bbbk[X]$.

The paper is structured as follows. The next section is devoted to review some facts about Rot--Baxter operators, relative Rot--Baxter operators, racks and quandles.
We give definition and prove some properties of relative Rota--Baxter operators on racks  in Section~\ref{RelRack}.

In Section \ref{Aver} we define semi-direct product of racks, that gives an answer on a question ~\cite[Problem 4.15]{BNS}, further we introduce  a notion
of (relative) averaging operators on racks. We proved that if  $(X, *)$ is a rack  and $B : X \to X$ is an  averaging operator, then it defines
 new algebraic  operation $\circ$ on $X$ such that $(X, \circ)$ is a rack,  $B$ is an   averaging operator on $(X, \circ)$ and $B$ induces a rack homomorphism $B : (X, \circ) \to (X, *)$.

In Section \ref{PropAver} we formulate some properties of averaging operators. In particular, for Cartesian product
$X:= X_1 \times X_2$ of quandles $X_1$ and $X_2$ we define   operators $P_i:   X \to X_i$ and prove that they are  averaging operators on $X$. Also we define averaging operator on union of averaging racks.

In Section \ref{RackRing} we consider two approach to definition of Rota--Baxter and averaging  operators on rack algebras. We find conditions under which linear extension of RB-operator on $X$ gives an RB-operator on algebra $\Bbbk[X]$. We prove that any averaging  operator on $X$ can be extended to an averaging operator on  $\Bbbk[X]$.

In Section \ref{GenRack} we  suggested some generalizations of conjugation and core  quandles on Rota--Baxter groups.

At the end of the article we formulate some questions for further research.


\section{Preliminaries} \label{Pr}

A {\it groupoid} $(G, \cdot)$  is a non-empty set $G$ with a binary algebraic operation `$\cdot$' (see \cite{Kurosh}).
A groupoid $G$ is called {\it idempotent} if $a \cdot a = a$ for any $a \in G$.

Let $\Bbbk$ be a field and let $A$ be an  algebra over  $\Bbbk$. A linear operator $R: A \to A$ is called
a {\it Rota--Baxter operator of weight $\lambda\in \Bbbk$} if
\begin{equation}\label{RBAlgebra}
R(x)R(y) = R( R(x)y + xR(y) + \lambda xy ), \qquad\qquad(x, y \in A).
\end{equation}
An algebra endowed with a Rota--Baxter operator is called a {\it Rota--Baxter algebra}.

In \cite{GLS} were defined   Rota--Baxter operators of weights $\pm1$ on a group.

\begin{definition}[\cite{GLS}] \label{Def:1}
Let $G$ be a group  and  let $B: G\to G$ be a  map.
\begin{itemize}
\item[(i)]  $B$ is called a {\it Rota--Baxter operator (RB-operator)  of weight 1 } if
\[
B(g)B(h) = B( g B(g) h B(g)^{-1} ),  \qquad\qquad(g,h\in G);
\]
\item[(ii)] $B$ is called a {\it Rota--Baxter operator (RB-operator)  of weight $-1$} if
\[
B(g) B(h) = B( B(g) h B(g)^{-1} g ),  \qquad\qquad(g,h\in G).
\]
\end{itemize}
\end{definition}

A group endowed with an RB-operator is called a {\it Rota--Baxter group (RB-group)}.

\begin{example}\label{Exccgaf}
Let $G$ be a group. The following maps:
\[
 B_{[0]}(g):= e; \qquad  B_{[-1]}(g): = g^{-1},\qquad \qquad (g\in G)
\]
are $RB$-operators  which are  called {\it  elementary Rota--Baxter operators}.
\end{example}

A group $G$ is called {\it Rota--Baxter-elementary}, if any Rota-Baxter operator on $G$ is elementary.
It is evident that every cyclic group of prime order is elementary.


\begin{proposition}[\cite{GLS}] \label{Prop:2}
 Let $(G, B)$ be a Rota--Baxter group. Define on  $G$ the operation
\[
g \circ h:= g B(g) h B(g)^{-1}, \qquad\qquad  (g, h \in G)
\]
The  following conditions hold:
\begin{itemize}
\item[(i)]   $(G, \circ)$ is  a group;

\item[(ii)] the map $B$ is an RB-operator on the group $(G, \circ)$;

\item[(iii)] $B:  (G, \circ) \to G$ is a group homomorphism of RB-groups  $(G, \circ, B)$ and  $(G, B)$.
\end{itemize}
\end{proposition}


\begin{definition}[\cite{JSZ}] \label{Def:2}
Let $H$ and $G$ be  groups.   Let   $\phi:  G\to \Aut(H)$ be a group homomorphism and let $B: H\to G$ be a map. A {\it relative Rota--Baxter group} is a quadruple $(H,G,\phi, B)$, where
\begin{equation}\label{eq:relative_group_RB}
B(h_1)B(h_2)=B(h_1\phi_{B(h_1)}(h_2)), \qquad\quad   (h_1, h_2\in H)
\end{equation}
 in which  $\phi_{g}:=\phi(g)\in \Aut(H)$ and  $g\in G$. The map $B$ is called a {\it relative RB-operator} on $H$ with respect to $(G, \phi)$.
\end{definition}

In particular, if $H=G$ and $\phi_g:  G \to \Inn(G)$, in which   $\phi_g(h):=ghg^{-1}$ for some fixed $g\in G$, we get a classical  Rota--Baxter group $(H, B)$.

The following definition can be found in the papers Joyce \cite{Joyce} and Matveev \cite{Matveev}.

\begin{definition}\label{Def:4jja}
A groupoid  $(Q, *)$ is called {\it quandle} if  the following axioms hold:
\begin{itemize}
\item[{\small (Q1)}] $x*x=x$ for all $x \in Q$;
\item[(Q2)] for each  $x,y \in Q$ there exists a unique $z \in Q$ such that $x=z*y$, in this case we write $z: = x \bar{*} y$;
\item[(Q3)] $(x*y)*z=(x*z) * (y*z)$ for all $x,y,z \in Q$.
\end{itemize}
If the groupoid  $(Q, *)$ satisfying only conditions (Q2) and (Q3), then  it  is called a {\it rack}.

\end{definition}

\begin{remark}
The rack (quandle)  defined in Definition \ref{Def:4jja} is called a {\it right rack (quandle)}.
If we take instead (Q2) and (Q3) the following two  axioms:
\begin{itemize}
\item[(LQ2)] for each  $x,y \in Q$ there exists a unique $z \in Q$ such that $x=y*z$, in this case we write $z: = y \bar{*} x$;
\item[(LQ3)] $z* (x*y)=(z * x) * (z *y)$ for all $x,y,z \in Q$;
\end{itemize}
then we get the definition of left rack (quandle).
\end{remark}

Used quandles and racks in  of our article are presented in the  following.
\begin{example}\label{Def:5Fvvja}
Let $(G,\cdot)$ be a group and $\phi \in \Aut(G)$. Let $C_n$ be a  cyclic group of order $n\geq 1$,
\begin{itemize}

\item[(i)] if  $a*_mb:= b^{-m} a b^{m}$ for any $a,b\in G$ and $m\in \mathbb{Z}$,  then  $\Conj_{m}(G):=(G,*_m)$  is  called the {\it $m$-conjugation quandle} of $G$. If $m=1$ then $\Conj(G):=\Conj_{1}(G)$ is called {\it conjugation quandle} of $G$;

\item[(ii)] if $a*b:= b a^{-1} b$  for any $a,b\in G$, then $\Core(G):=(G,*)$ is  called the {\it core quandle} of $G$. In particular, the core quandle $\R_n:=(C_n,*)$    is called the {\it dihedral quandle};
\item [(iii)] if $a*b:=  \phi(ab^{-1})b$ for any $a,b\in G$, then  $\Alex(G,\phi):=(G,*)$ is called the   {\it generalized Alexander quandle} of $G$ with respect to $\phi$.
\end{itemize}
\end{example}

A quandle  $(Q,*)$ is called {\it trivial} if $x*y=x$ for all $x, y \in Q$.  Unlike groups, a trivial quandle can have arbitrary number of elements. We denote the $n$-element trivial quandle by $\mathrm{T}_n$ and an arbitrary trivial quandle by $\mathrm{T}$.

If  $(Q, *)$ is a  quandle, then from axiom (Q3) we obtain that
\[
(x * y) * x = x * (y * x),\qquad\qquad(x, y \in Q).
\]
Furthermore,  the axioms (Q2) and (Q3) are equivalent to existing   the map
\begin{equation}\label{FLVVr}
S_x: Q \to Q\qquad (y\mapsto y*x,\quad y\in Q)
\end{equation}
being an automorphism of $Q$ for each $x \in Q$. These automorphisms are called {\it inner automorphisms}, and the group generated by all such automorphisms is denoted by $\Inn(X)$. A quandle  $Q$ is called {\it connected}  if it admits a transitive action by its group of inner automorphisms. For example, dihedral quandles of odd order are connected, whereas that of even order are disconnected. A quandle $X$ is called {\it involutary} if $S_x^2 = \id_Q$ for each $x \in Q$. For example, all core quandles are involutary.  A quandle (resp. rack) $(Q,*)$ is called {\it commutative} if $x * y = y * x$ for all $x, y \in Q$. The dihedral quandle $\R_3$ is commutative and no trivial quandle $\mathrm{T}_n$ with $n>1$ (i.e. more than one element) is commutative.


\section{Relative Rota--Baxter operators on racks} \label{RelRack}


 To define a   relative RB-operators on racks we recall  construction  of relative RB-operators on groups (see \cite{JSZ}).

Let $H$ and $G$ be groups  and let $\Phi :   G \to \Aut(H)$ be a homomorphism.  It is easy to check that  $(G \times H, * )$ is  a  group in which  the  multiplication $*$ is defined by
\begin{equation}\label{YbbDD}
(g_1, h_1) * (g_2, h_2):=  \left( g_1 g_2, (h_1 \Phi_{g_2}) h_2 \right),\qquad (g_1,g_2\in G, h_1,h_2\in H)
\end{equation}
where  $h_1 \Phi_{g_2}$ is the image of the element $h_1$ under the action of the automorphism $\Phi_{g_2}:= \Phi(g_2)$.
The group $G \ltimes_{\Phi} H$  is called  a {\it semi-direct product} of $G$ and $H$ with respect to $\Phi$. In particular, if $G = \Aut(H)$ then the group $\Hol(H):=\Aut(H) \ltimes_{\Phi} H$  is called  the  {\it holomorph}  of the group $H$.

Now let $B:   H \to G$ be a map and let
\[
\Gamma_B: = \{ (B(h), h)\mid h \in H \}\subseteq G \times H
\]
be its graphic. The groupoid  $(\Gamma_B,*)$   in  which
\[
(B(h), h) * (B(k), k):=  \left(B(h) B(k), (h \Phi_{B(k)}) k\right), \qquad (h,k\in H)
\]
is a group  if and only if $B$ satisfies the following analogous of the relative RB-operator
\begin{equation} \label{ORRBG}
B(h) B(k) = B\left((h \Phi_{B(k)}) k \right), \qquad\quad   (h, k \in H).
\end{equation}

\begin{remark}
The original definition of relative RB-operator on groups (see \eqref{eq:relative_group_RB}) does not coincide with   \eqref{ORRBG} because  we have a different multiplication in a semi-direct product.  If we define the multiplication on $G \ltimes_{\Phi} H$ by the formula
$$
(g_1, h_1) * (g_2, h_2):=  \left( g_1 g_2, h_1 \cdot \Phi_{g_1}(h_2) \right),\qquad (g_1, g_2 \in G,\;  h_1, h_2\in H),
$$
we get the relative RB-operator \eqref{eq:relative_group_RB}. Our approach is motivated by the using  right racks and quandles.
\end{remark}

\bigskip

Using  this construction we  define relative Rota--Baxter operators on racks.
\smallskip

Let $(X, \cdot)$ and $(A, \circ)$ be  racks. Let   $\Phi:   A \to \Conj(\Aut(X))$ be a rack homomorphism such that  $a \mapsto \Phi_a \in \Aut(X)$. Define on the set
$A \times X$ the following multiplication:
\[
(a, x) * (b, y):= (a \circ b, (x \Phi_{b}) \cdot y ),\qquad (x, y \in  X, a,  b \in A).
\]
Further, let $B:   X \to A$ be a map  and let
$$
\Gamma_B:= \{ (B(x), x)\mid x \in X \} \subset A \times X
$$
be its graphic.
The multiplication on $A \times X$ induces multiplication on $\Gamma_B$:
\[
(B(x), x) * (B(y), y):=  \left(B(x) B(y), (x \Phi_{B(y)}) y\right).
\]
If $B(x) B(y) = B\left((x \Phi_{B(y)}) y \right)$, then the set $\Gamma_B$ is closed under multiplication $*$ and we get a groupoid $(\Gamma_B,*)$.

Now we are ready to introduce the following.

\begin{definition}\label{Def:6.1}
Let $(X, \cdot)$ and $(A, *)$ be  racks  and let $\Phi:   A \to \Conj(\Aut(X))$ be a rack homomorphism. A {\it relative  Rota--Baxter operator} (RRB-operator) on  $(X, \cdot)$ with respect to $(A, \Phi)$ is a map $B:   X \to A$ such that
\begin{equation}\label{ZZxG2355H}
B(x) * B(y) = B\left((x \Phi_{B(y)}) \cdot y \right), \qquad\qquad(x, y \in X)
\end{equation}
where $\Phi_{B(y)}:= \Phi(B(y))$ is the image of $B(y) \in \Aut(X)$. The quadruple $(X, A, \Phi, B)$ is called a {\it relative Rota-Baxter rack} (RRB-rack).
If $X$ and $A$ are quandles, then the quadruple $(X, A, \Phi, B)$ is called a {\it relative Rota-Baxter quandle} (RRB-quandle).
\end{definition}

If $\Phi = {\bf 1}$ is the identity map (i.e.  $\Phi(a) = \Phi_a = \id \in \Aut(X)$ for any $a \in A$), then from~\eqref{ZZxG2355H} we have $B(x) * B(y)=B\bigl(x \cdot y \bigr)$ for any $x, y \in X$.

This yields the  following.

\begin{lemma}\label{Prop:3}
Let $(X, \cdot)$ and $(A, *)$ be  two racks. If  $f:   X \to A$ is a rack homomorphism, then the quadruple $(X, A,{\bf 1}, f)$ is a relative Rota--Baxter rack.
\end{lemma}

If $A = X$ and $\Phi:   X \to \Conj(\Inn(X))$, then we introduce the following.

\begin{definition}\label{Def:11}
Let $(X, *)$ be a  rack. A map $B:   X \to X$  such that
\[
B(x) * B(y) = B\left((x * B(y)) * y \right), \qquad\qquad(x, y \in X)
\]
is called  a {\it  Rota--Baxter operator} (RB-operator) on $X$ and $(X, B)$    is
 called an  {\it RB-rack}.
 \end{definition}

\begin{remark}
If $(X, *)$ is a left rack, then one can  define an {\it  RB-operator} on $X$ as a map $B:   X \to X$  such that
\[
B(x) * B(y) = B \left( x * (B(x) * y) \right), \qquad\qquad(x, y \in X).
\]
\end{remark}

Using these properties of RRB-racks we obtain the following.

\begin{theorem} \label{Theor:7.5}
Let $(X, \cdot)$ and $(A, *)$ be two racks  and let $\Phi:   A \to \Conj(\Aut(X))$ be a rack homomorphism,  $a \mapsto \Phi_a \in  \Aut(X)$.  Let $(X, A, \Phi, B)$ be an   RRB-rack. Set
\[
x \circ y:=  (x \Phi_{B(y)})  \cdot y \in X, \qquad\qquad(x, y \in X).
\]

\begin{itemize}
\item[(i)] If $\im \Phi = \{\id\}$, then
\[
B(x) * B(y) = B(x \cdot y),   \qquad\qquad  (x, y \in X).
\]
Hence, $B : X \to A$ is a homomorphism and $x \circ y = x \cdot y$.

\item[(ii)] If $\im \Phi \not= \{\id\}$ and  for any $x, y, z \in X$  holds
$$
\bigl((x \Phi_{B(y)})\Phi_{B(z)}\bigr) \cdot z = \Big(\bigl(x \Phi_{B(z)} \bigr) \Phi_{B((y \Phi_{B(z)}) \cdot z)}\Big) \\
\cdot \bigl(z  \Phi_{B((y \Phi_{B(z)}) \cdot z)}),
$$
then $(X, \circ)$ is a rack.

Moreover, if $(X, A, \Phi, B)$ is a RRB-quandle  and  $x \Phi_{B(x)}=x$  for any $x \in X$, then
$(X, \circ)$ is a  quandle.
\end{itemize}
\end{theorem}

\begin{proof} (i) Is evident.

(ii) Let us check the  quandle axioms for the groupoid $(X, \circ)$. To get (Q1), we have to show that
\[
x \circ x = (x  \Phi_{B(x)}) \cdot x = x,   \qquad\qquad  (x \in X)
\]
which  is possible only if    $x \Phi_{B(x)}= x$,  i.e.  the automorphism $\Phi_{B(x)}$ fixes each $x \in X$.

Since $\Phi_{B(a)}$ is an automorphism and $(X, \cdot)$ is a rack, for any $a, b \in X$ there exists a unique $x:= \Phi^{-1}_{B(a)}(b \bar{\cdot} a) \in X$ such that $x \circ a = x  \Phi_{B(a)}) \cdot a  = b$. Hence  (Q2) holds.

Now, by direct calculation we have
\[
\begin{split}
(x \circ y) \circ z &= \bigl((x \Phi_{B(y)}) \cdot y\bigr)  \circ z \\
&= \bigl(((x \Phi_{B(y)}) \cdot y\bigr) \Phi_{B(z)}) \cdot z\qquad\qquad  \mbox{($\Phi_{B(z)}$ is an automorphism)}\\
& =  \bigl((x \Phi_{B(y)})\Phi_{B(z)} \cdot (y \Phi_{B(z)})\bigr) \cdot z \\
& = \bigl(((x \Phi_{B(y)})\Phi_{B(z)}) \cdot z \bigr) \cdot \bigl((y \Phi_{B(z)}) \cdot z \bigr);\\
(x \circ z) \circ (y \circ z) &= \bigl((x \Phi_{B(z)}) \cdot z\bigr)  \circ \bigl((y \Phi_{B(z)}) \cdot z\bigr)\\
 &= \bigl(((x \Phi_{B(z)}) \cdot z\bigr) \Phi_{B((y \Phi_{B(z)}) \cdot z)}) \cdot \bigl((y \Phi_{B(z)}) \cdot z\bigr).
\end{split}
\]
Using the right cancellation (axiom (Q2)), it follows that
\[
\begin{split}
\bigl((x \Phi_{B(y)})\Phi_{B(z)}\bigr) \cdot z&= \bigl((x \Phi_{B(z)}) \cdot z\bigr) \Phi_{B((y \Phi_{B(z)}) \cdot z)}\\
&= \bigl(x \Phi_{B(z)} \bigr) \Phi_{B((y \Phi_{B(z)}) \cdot z)}) \cdot \bigl(z  \Phi_{B((y \Phi_{B(z)}) \cdot z)})
\end{split}
\]
because  $\Phi_{B((y \Phi_{B(z)}) \cdot z)}$ is an automorphism. Consequently,   (Q2) holds. \end{proof}

\begin{remark}
From the proof follows that the equality in (ii) is equivalent to the equality
\[
\bigl((x \Phi_{B(y)})\Phi_{B(z)}\bigr) \cdot z = \bigl((x \Phi_{B(z)}) \cdot z\bigr) \Phi_{B((y \Phi_{B(z)}) \cdot z)}
\]
\end{remark}

From  Theorem \ref{Theor:7.5} we have the following.

\begin{corollary} \label{Cor}
Let $(X, \cdot, B)$ be an  RB-rack. Define on  $X$ an operation:
\[
x \circ y := (x \cdot  B(y)) \cdot y,   \qquad\qquad  (x,y  \in X).
\]
If  for any $x, y, z \in X$  holds
\[
\begin{split}
\bigl((x \cdot B(y)) & \cdot B(z)\bigr) \cdot z \\
&= \bigl((x \cdot B(z)) \cdot z \bigr) \cdot B((y \cdot B(z)) \cdot z),
\end{split}
\]
then $(X, \circ)$ is a rack.

Moreover, if $(X, \cdot, B)$ is an  RB-quandle  and for any $x \in X$ holds  the equation $x \cdot B(x)=x$, then
the groupoid $(X, \circ)$ is  a  quandle.
\end{corollary}

Some RB-operators  are presented in the following.

\begin{example} \label{Prop:5}
\begin{itemize}
\item[(i)] Let $X$ be a quandle,    $p \in X$.  Then the map $B_p(x):= p$\quad  ($x \in X$)  is an RB-operator.

\item[(ii)] For  a trivial quandle $\T$ any map $B:   T \to T$ is  an  RB-operator.
\end{itemize}
\end{example}

The next properties of  RB-operators  are evident.

\begin{proposition} \label{EvPr}
Let $(X, \cdot,  B)$ be a  Rota--Baxter rack (quandle). Then
\begin{itemize}
\item[(i)] the image $\im B$ of $B$ is a rack (quandle);

\item[(ii)] if  $(X_1,B_1)$ and $(X_1,B_1)$ are  two  Rota--Baxter  racks, then the map
\[
B((x_1, x_2)):=  (B_1(x_1), B_2(x_2)), \qquad (x_i \in X_i)
\]
is a  Rota--Baxter operator on $X: = X_1 \times X_2$.
\end{itemize}
\end{proposition}


\section{Rack extensions and averaging operators on racks} \label{Aver}

In the first part of the present section we suggest some approach to solution of the problem of a description  semi-direct product of quandles (see  \cite[Problem 4.15]{BNS}).   A particular case of this construction was introduced in \cite{Kin}.  Note that,  rack extensions were investigated in  \cite{AG, BCK, BS, CENS, J}.

\begin{theorem} \label{rack}
Let $(X, \cdot)$ and  $(A, *)$ be  two racks and let $\Phi:   A \to \Conj(\Aut(X))$ be a rack homomorphism. On the set $Y:= A \times X$  we  define the following operation:
\[
(a, x) \circ (b, y):= (a * b, x \Phi_b),\qquad\qquad  (a, b \in A, x, y \in X)
\]
in which  $x \Phi_b$ means the action of the automorphism $\Phi_b =  \Phi (b)$ on the element $x$.

The gropoid  $(Y, \circ)$ is a rack. Additionally, if  $X$ and $A$ are quandles and $\Phi_a$ lies in the kernel of $\Phi$  for any $a \in A$, then $(Y, \circ)$ is a quandle.
\end{theorem}

\begin{proof}  Let us  check the rack axioms. The equation $u \circ v = w$ with unknown $u = (a, x) \in Y$ and given  $v = (b, y), w= (c, z) \in Y$ has a unique solution. Indeed, it  has the  form
\[
(a * b, x \Phi_b) = (c, z).
\]
The element   $a = c \bar{*} b$ is a unique solution for the equation $a * b = c$, because $A$ is a rack. Furthermore,  $\Phi_b$ is an automorphism, so the equation $x \Phi_b = c$ has a unique solution $x = c \Phi^{-1}_b$. Hence  the axiom (Q2) holds.

Now, it is easy to check that
\[
\begin{split}
(u \circ v) \circ w &= (a * b, x \Phi_b) \circ (c, z) = \left((a * b) * c, (x \Phi_b) \Phi_c \right),\\
(u \circ  w) \circ (v \circ w) &= (a * c,  x \Phi_c) \circ (b * c,  y \Phi_c) = \left( (a * c) * (b * c),  (x \Phi_c) \Phi_{b*c} \right),
\end{split}
\]
so we have  to prove that $(u \circ v) \circ w = (u \circ  w) \circ (v \circ w)$ for any $u = (a, x)$, $v = (b, y)$, $w = (c, z) \in Y$.
Since $A$ is a rack,  holds the equality $(a * b) * c =  (a * c) * (b * c)$. Using the fact that  $\Phi$ is a rack homomorphism to a conjugacy quandle,  we obtain  that
\[
\Phi_{b*c} = \Phi_c^{-1} \Phi_b \Phi_c\qquad \text{and}\qquad  (x \Phi_b) \Phi_c = (x \Phi_c) \Phi_{b*c},
\]
so   (Q3) holds.

If $X$ and $A$ are quandles and $\Phi_a$ lies in the kernel of $\Phi$, then  $u \circ u = (a * a, x \Phi_a) = (a, x\; \id) = u$, so the axiom  (Q1) holds.
\end{proof}

\begin{definition} \label{Def:7NNNzaq}
Let the conditions of Theorem \ref{rack} hold.
The rack $A \ltimes_{\Phi} X:=( A\times X, \circ)$ is called  a {\it semi-direct product of $X$ and $A$ with respect to $\Phi$}.

In particular, if $A = \Conj(\Aut(X))$ and $\Phi = \id$, then $\Hol(X):=\Conj(\Aut(X)) \ltimes_{\Phi} X$ is called  the {\it holomorph} of $X$.
\end{definition}

Let us define a relative averaging operator on racks.

\begin{definition}\label{Def:8:TTcvdE}
Let $(X, \cdot)$ and $ (A, *)$ be  racks and let $\Phi:   A \to \Conj(\Aut(X))$ be a rack homomorphism.
A map $B:   X \to A$ is called  a {\it relative averaging operator} on $X$ with respect to $(A, \Phi)$, if
\[
B(x) * B(y) = B(x \Phi_{B(y)}),\qquad (x, y \in X)
\]
in which $x \Phi_{B(y)}$ is the image of the element $x$ under the action of the automorphism $\Phi_{B(y)}$. The quadruple $(X, A, \Phi, B)$ is called a {\it relative averaging rack} (RA-rack).
\end{definition}

An analogue of the corresponding theorems for Rota--Baxter operators and relative Rota--Baxter operators for groups (see \cite{GLS, JSZ}) is the following.

\begin{theorem}\label{MMgd}
Let $(X, \cdot)$ and  $(A, *)$ be  two racks and let $\Phi:   A \to \Conj(\Aut(X))$ be a rack homomorphism. Let $B: X \to A$ be a map and let
\[
\Gamma_B:= \{(B(x), x)\mid x \in X \} \subseteq A \ltimes_{\Phi} X.
\]
The map $B: X \to A$ is   a relative averaging operator on $X$ with respect to $(A, \Phi)$ if and only if $\Gamma_B$  is a subrack of $A \ltimes_{\Phi} X$.
\end{theorem}
\begin{proof} Evidently,  $\Gamma_B$ is a subset of $A \ltimes_{\Phi} X$, so   we have
\[
(B(x), x) \circ (B(y), y):= (B(x) * B(y), x \Phi_{B(y)}),\qquad (x, y\in X).
\]
But   $(B(x), x) \circ (B(y), y)\in  (\Gamma_B,\circ)$  if and only if $B(x) * B(y) = B(x \Phi_{B(y)})$ for any $x, y \in X$.
\end{proof}

If $X = A$ and $\Phi : X \to \Inn(X)$ such that $\Phi_{B(y)} = S_{B(y)}$ is the inner automorphism of $X$, then we define the following.

\begin{definition}\label{Def:9KKjfm}
Let $(X, \cdot)$  be  a rack and let     $B: X \to X$ be a map. If
\[
B(x) \cdot B(y) = B(x \cdot B(y)),\qquad (x, y \in X)
\]
then $B$ is called  a  {\it right  averaging operator} on $X$.  If
\[
B(x) \cdot B(y) = B(B(x) \cdot y),\qquad (x, y \in X)
\]
then $B$ is  called a {\it left averaging operator} on $X$.
\end{definition}

In the sequel of our  article   we  mainly consider  right averaging operators which is   called simply {\it averaging operators}.   A pair $(X, B)$ is called an {\it right (left) averaging rack}, if  $X$ is a rack and $B$ is a right (left) averaging operator. We call a right  averaging rack simply an averaging rack.

\begin{remark}
In  racks the multiplication on the right is an automorphism,  but multiplication on the left, in general, is not. Consequently,   the theory of   right averaging operators is different from the theory of left averaging operators. It is interesting to construct a theory of left averaging operators.
\end{remark}

Some properties of averaging operators on racks are similar to the properties of RB-operators on groups (see Proposition \ref{Prop:2}), which shows the following.

\begin{theorem} Let   $(X,  B)$ be  an averaging  rack and
let   $(X, \circ_B)$ be a  groupoid, in which
\begin{equation}\label{QQQvf}
x \circ_B y := x \cdot B(y),\qquad ( x, y \in X).
\end{equation}
The following conditions hold:
\begin{itemize}
\item[(i)]  $(X, \circ_B)$ is a rack. Moreover, if $x \cdot B(x) = x$ for all $x \in X$, then $(X, \circ_B)$ is a quandle;

\item[(ii)]  $B$ is an averaging operator on the rack $(X, \circ_B)$;

\item[(iii)] the map  $B:   (X, \circ_B) \to (X, \cdot)$ is a  homomorphism.
\end{itemize}
\end{theorem}

\begin{proof}
(i)  Let us check the rack axioms.
The equation $x \circ_B B(a) = b$ has a unique solution $x = b \, \bar{\cdot} \, B(a)$, i.e.  (Q2) holds.

From the equations
\[
\begin{split}
(x \circ_B y) \circ_B z &= (x \cdot B(y)) \circ_B z = (x \cdot B(y)) \cdot B(z) = (x \cdot B(z)) \cdot (B(y) \cdot B(z))\\
& =  (x \cdot B(z)) \cdot (B(x \cdot B(y));\qquad (\text{by definition of $B$})\\
(x \circ_B z) \circ_B (y \circ_B z) &= (x \cdot B(z)) \circ_B (y \cdot B(z)) = (x \cdot B(z)) \cdot B(y \cdot B(z))
\end{split}
\]
follows that $(x \circ_B y) \circ_B z = (x \circ_B z) \circ_B (y \circ_B z)$ for any $x, y, z \in X$, so  (Q3) holds.

Further, by the assumption $x \circ_B x = x \cdot B(x) = x$, so  (Q1) holds.

\noindent
(ii) An averaging operator $B$ on $(X, \circ_B)$  must satisfies the identity
\[
B(x) \circ_B B(y) = B(x \circ_B B(y)),\qquad (x, y \in X)
\]
which is by \eqref{QQQvf} has the form
\[
B(x) \cdot B\left(B(y)\right) = B\left(x \cdot B\left(B(y)\right) \right), \qquad (x, y \in X).
\]
If we put $z = B(y)$, we get $B(x) \cdot B\left(z \right) = B\left(x \cdot B(z) \right)$.
Since $B$ is an averaging operator on $(X, \cdot)$, this identity holds.

\noindent
(iii) By the definition of $B$ and operation $\circ_B$, we have
\[
B(x) \cdot B\left(y \right) = B\left(x \circ_B y \right),\qquad (x, y \in X).
\]
It is means that $B$ is a homomorphism. \end{proof}

\begin{problem}
Let   $(X, \cdot, B)$ be  an averaging  rack. Then $(X, \circ_B)$ is a rack, where $\circ_B$ as in  \eqref{QQQvf}.  The set  $(X,\cdot, \circ_B)$ with two operations $\cdot$ and $\circ_B$,    is an  analogue of a skew brace (see~\cite{GV}). Which axioms connect the rack operations $\cdot$ and $\circ_B$?
\end{problem}


\section{Properties of averaging operators on racks} \label{PropAver}

The next properties of averaging operators are evident.

\begin{lemma} \label{EvPr}
Let $(X, \cdot,  B)$ be an averaging rack (quandle).
\begin{itemize}
\item[(i)] the image $\im B$ of $B$ is a rack (quandle);

\item[(ii)] if  $(X_1,B_1)$ and $(X_1,B_1)$ are  two averaging  racks, then the map
\[
B((x_1, x_2)):=  (B_1(x_1), B_2(x_2)), \qquad (x_i \in X_i)
\]
is an averaging operator on $X: = X_1 \times X_2$.
\end{itemize}
\end{lemma}

Let $X_1$ and $X_2$ be quandles  and let $X_1 \times X_2 := \{ (x', x'')\mid x' \in X_1, x'' \in X_2 \}$
be their  Cartesian  product with a component-wise multiplication. To define operators projectivizations, we fix elements  $y_0 \in X_1$, $z_0 \in X_2$, and define the following maps:
\begin{equation}\label{YYcdJJ}
\begin{split}
P_1&:  X_1 \times X_2 \to  X_1 \times X_2, \quad P_1\big((x', x'')\big) = (x', z_0); \\
P_2&:  X_1 \times X_2 \to  X_1 \times X_2, \quad P_2\big((x', x'')\big) = (y_0, x'').
\end{split}
\end{equation}
Clearly, each  $P_i(X) \cong X_i$  and  does not depend on the choice of  $y_0$ and $z_0$, respectively.

\begin{remark}
 In the case of racks, the image  $P_i(X)$ can be non-isomorphic to $X_i$.
\end{remark}

If  $G = G_1 \times G_2$ is a direct product of groups $G_1$ and $G_2$, then the projections $P_i:   G \to G_i$ ($i=1, 2$) are Rota--Baxter operators. A similar result for quandles is the following.

\begin{proposition}
Let $X_1$ and $X_2$ be quandles and let $X:= X_1 \times X_2$  be their  Cartesian product with a component-wise multiplication. Each  operator $P_i:   X \to X_i$  (see \eqref{YYcdJJ}) is an  averaging operator on $X$.
\end{proposition}

\begin{proof}
Consider the case of the operator $P_1$.  If $x = (x_1, x_2),  y= (y_1, y_2) \in X$, then
\[
P_1(x) * P_1(y) = P_1 (x * P_1(y)).
\]
Indeed, if   $z_0 \in X_2$ is our fixed  element, then
\[
\begin{split}
P_1(x) * P_1(y) &=  (x_1, z_0) * (y_1, z_0) = (x_1 * y_1, z_0);\\
P_1 (x * P_1(y))&=  P_1((x_1, x_2) * (y_1, z_0)) = P_1((x_1 * y_1, x_2 * z_0)) = (x_1 * y_1,  z_0).
\end{split}
\]
The case of the operator $P_2$ is similar to the first one.
\end{proof}


 Recall a construction of disjoint union of racks (see, for example, \cite[Proposition 9.2]{BarTimSin}).


\begin{definition}
Let $(X_1, *)$, $(X_2, \circ)$ be racks and let ${\rm Conj}_{-1}\left({\rm Aut}(X_1) \right)$, ${\rm Conj}_{-1}\left({\rm Aut}(X_2) \right)$ be   $(-1)$-conjugation quandles of $X_1$ and $X_2$, respectively (see Example \ref{Def:5Fvvja}(i)). Let
$$
\sigma:   X_1 \to  {\rm Conj}_{-1} \left({\rm Aut}(X_2) \right), ~~\tau:   X_2 \to  {\rm Conj}_{-1} \left({\rm Aut}(X_1) \right)
$$
 be rack homomorphisms. The set $X:=X_1 \sqcup X_2$ with the operation:
\[
x\star y:=\begin{cases}
x*y,& x, y \in X_1; \\
x\circ y,  &x, y \in X_2; \\
{\tau(y)}(x),  &x \in X_1, y \in X_2; \\
{\sigma(y)}(x), &x \in X_2, y \in X_1,
\end{cases}
\]
is a rack if and only if the following conditions hold:
\begin{itemize}
\item[(i)] $\tau(z)(x)* y=\tau\left(\sigma(y)(z)\right)(x* y)$; \qquad ($x, y \in X_1$ and $z \in X_2$)
\item[(ii)] $\sigma(z)(x)\circ y=\sigma\left(\tau(y)(z)\right)(x\circ y)$, \qquad  ($x, y \in X_2$ and $z \in X_1$).
\end{itemize}
The   rack $X_1\underset{\sigma,\tau}{\sqcup} X_2:=(X_1\sqcup X_2,\star)$  is  called a {\it union of racks} $X_1$ and  $X_2$ (with respect to $\sigma,\tau$).
\end{definition}
For simplicity we  write $\tau_y(x)$ instead  $\tau(y)(x)$ and $\sigma_y(x)$ instead  $\sigma(y)(x)$, respectively.

A subrack $Y$ of a rack $(X,*)$ is called {\it  normal}  if $y * x\in Y$ for all  $y \in Y$ and  $x\in X$. Evidently,  if $Y$ is a normal subrack of $X$, then $Y$ is a union of orbits in $X$. Racks $X_1, X_2$ are obviously normal in $X=X_1\underset{\sigma,\tau}{\sqcup} X_2$, therefore if $X$ is  a connected rack, then it cannot be presented as a union of two subracks. If $X$ is not connected, $X_1$ is an orbit of some element from $X$, $X_2=X\setminus X_1$, then $X=X_1\underset{\sigma,\tau}{\sqcup} X_2$ for appropriate $\sigma,\tau$. So, a rack $X$ can be presented as a union of two racks if and only if $X$ is not connected. The dihedral quandle ${\rm R}_{2n}$ of an even order can be presented as a union ${\rm R}_n\underset{\sigma,\tau}{\sqcup} {\rm R}_n$  of two subquandles (see \cite{BarTimSin}) for appropriates  $\sigma, \tau$.


\begin{theorem}
Let  $(X_1, B_1)$ and  $(X_2, B_2)$ be racks with the averaging operators $B_1$ and $B_2$, respectively. Let $X=X_1\underset{\sigma,\tau}{\sqcup} X_2$ for some $\sigma, \tau$.   Let  $B:   X\to X$ be a map in which
\[
B(z):=\begin{cases}
B_1(z),  &\qquad\text{for}\quad  z \in X_1; \\
B_2(z),  &\qquad\text{for}\quad z \in X_2.
\end{cases}
\]
The map $B$ is an averaging operator if and only if for any $z_1 \in X_1$  and $z_2 \in X_2$ hold the following:
\[
\begin{split}
\tau_{B_2(z_2)}(B_1(z_1)) &= B_1(\tau_{B_2(z_2)}(z_1));\\
\sigma_{B_2(z_1)}(B_1(z_2)) &= B_1(\sigma_{B_2(z_1)}(z_2)).
\end{split}
\]
\end{theorem}

\begin{proof}
Let us  check the identity
\begin{equation}\label{JNNyyym}
B(z_1) \star B(z_2) = B(z_1 \star B(z_2)), \qquad (z_1, z_2 \in X).
\end{equation}
If $z_1, z_2\in X_1$, or $z_1, z_2\in X_2$, then \eqref{JNNyyym} follows from the definition of $B$.
If $z_1 \in X_1, z_2 \in X_2$, then
\[
B_1(z_1) \star B_2(z_2) = B_1(z_1 \star B_2(z_2)),
\]
which is  equivalent to $\tau_{B_2(z_2)}(B_1(z_1)) = B_1(\tau_{B_2(z_2)}(z_1))$.

Let  $z_1 \in X_2$ and  $z_2 \in X_1$. It is easy to check that $B_2(z_1) \star B_1(z_2) = B_2(z_1 \star B_1(z_2))$, which  is equivalent to
$\sigma_{B_2(z_2)}(B_1(z_1)) = B_1(\sigma_{B_2(z_2)}(z_1))$.\end{proof}



\section{Rota--Baxter and averaging operators on rack algebras} \label{RackRing}

 Let $A$ be an algebra  over a field $\Bbbk$.  As we seen in Section \ref{Pr}, the  RB-operator of weight $\lambda \in \Bbbk$ on $A$ is  a linear map $R : A \to A$ such that
\begin{equation}\label{GbbxCC}
R(a) R(b) = R (R(a) b + a R(b) + \lambda ab),\qquad \qquad (a, b \in A).
\end{equation}
Further, we will call this operator   $R$ by {\it algebraic Rota--Baxter operator} (shortly ARB-operator) of weight $\lambda$ on $A$.

\begin{example} \label{ARO}
For any algebra $A$,  the zero operator: $R(a) = 0$ ($a \in A$); and the operator: $R(a) = - \lambda a$ ($a \in A$), are algebraic Rota--Baxter operators. These operators are called {\it trivial Rota--Baxter operators}.
\end{example}

 Let  $B$ be an RB-operator on a rack $X$ and $A = \Bbbk[X]$ be  its  rack algebra over a field $\Bbbk$.  We can extend $B$ by linearity to  $A$:
\begin{equation}\label{GEEsCCC}
B\left( \sum_{x \in X} \alpha_x \, x \right): = \sum_{x \in X} \alpha_x \, B(x), \qquad\qquad  (\alpha_x \in \Bbbk, x\in X).
\end{equation}
We will call such operators   by {\it rack Rota--Baxter operators} on $A$. If $X$ is a quandle,  we will call such operators  by {\it quandle Rota--Baxter operators}.

\begin{remark}
Further we denote a RB-operator on a rack and its extension on rack algebra by the same symbol. It is clear from the context where this operator acts.
\end{remark}

Also, on  rack algebra there exist   so called monomial operators, who occupy the middle position between algebraic and rack RB-operators.

\begin{definition}
An algebraic RB-operator   $R: \Bbbk[X] \to \Bbbk[X]$ which acts by the formula
\[
R\left( \sum_{x \in X} \alpha_x \, x \right) = \sum_{x \in X} \alpha_x \, R(x)=\sum_{x \in X} (\alpha_x \beta_x) y_x,   \qquad\qquad  (\alpha_x, \beta_x \in \Bbbk, y_x \in X)
\]
is said to be  a {\it monomial} RB-operator on $\Bbbk[X]$.
\end{definition}

Now we are considering the following question. Let $(X,*)$ be a rack and let  $B:   X \to X$ be  an RB-operator  on $X$. Under which conditions $B$ induces an algebraic RB-operator on  $\Bbbk[X]$?

An answer  gives.

\begin{proposition} \label{Condition}
Let $(X,*)$ be a rack and   $B:   X \to X$ be  an $RB$-operator  on $X$.
The  linear extension of   $B$ to $\Bbbk[X]$ is an algebraic RB-operator of weight $ \lambda$ (i.e. satisfy \eqref{GbbxCC})
if and only if $ \lambda = -1$ and for any $x, y \in X$ we have either
\[
\begin{cases}
B(x) *  B(y) &= B (B(x) * y); \\
B(x* y)&=B(x * B(y)),\\
\end{cases}
\qquad\text{or}\qquad
\begin{cases}
B(x) *  B(y) &= B (x * B(y));   \\
 B(x* y)&=B (B(x) * y).\\
\end{cases}
\]
\end{proposition}

\begin{proof}
Let us show that the equality \eqref{GbbxCC} holds if  it holds for elements of $X$. Indeed, if
$u = \sum_{x \in X} \alpha_x \, x$ and  $v = \sum_{y \in X} \beta_y \, y\in \Bbbk[X]$ in which $\alpha_x, \beta_y \in \Bbbk$, then   using \eqref{GEEsCCC} we have
$$
B(u) *  B(v) = \sum_{x, y \in X} \alpha_x \beta_y \, \left( B(x) * B(y)\right),
$$
$$
B (B(u) * v) + B(u * B(v)) + \lambda B(u * v) = B\left( \sum_{x, y \in X} \alpha_x \beta_y \,  (B(x) * y) \right) +
B\left( \sum_{x, y \in X} \alpha_x \beta_y \, (x * B(y)) \right) +
$$
$$
+\lambda \left( \sum_{x, y \in X} \alpha_x \beta_y \, \left( B(x) * B(y)\right) \right).
$$
Hence, if for any $x, y \in X$ holds \eqref{GbbxCC}, then it holds for $u,v$.

Let us find conditions under which  the equality \eqref{GbbxCC} holds. Since all elements $B(x)$, $B(y)$ lies in $X$, the left hand side $B(x) *  B(y)$ of  \eqref{GbbxCC} must be equal to
$$
B (B(x) * y)~\mbox{or}~B(x * B(y))~\mbox{or}~\lambda B(x* y).
$$
But the last case is not possible since $B (B(x) * y) + B(x * B(y)) \not= 0$. Hence,
 \[
\begin{cases}
B(x) *  B(y) &= B (B(x) * y); \\
-\lambda B(x* y)&=B(x * B(y)),\\
\end{cases}
\qquad\text{or}\qquad
\begin{cases}
B(x) *  B(y) &= B (x * B(y));   \\
-\lambda B(x* y)&=B (B(x) * y),\\
\end{cases}
\]
which  is possible only  if $\lambda = -1$.
\end{proof}

\begin{proposition}
Let $\Bbbk[X]$ be the quandle algebra of a quandle $X$ over a field $\Bbbk$.
Let $B : \Bbbk[X] \to \Bbbk[X]$ be a quandle and  algebraic RB-operator on the quandle algebra $\Bbbk[X]$.  Then for any $ x \in X$ holds
$$
B (B(x) * x) = B (x * B(x)).
$$
\end{proposition}

\begin{proof}
Since $B$ is a quandle and algebraic RB-operator, one of the systems  from Proposition \ref{Condition} holds. Suppose that the first system holds. Then for $x = y$ we have
 \[
\begin{cases}
B(x) *  B(x) &= B (B(x) * x), \\
 B(x* x)&=B(x * B(x)),\\
\end{cases} \Leftrightarrow
\begin{cases}
B(x)  &= B (B(x) * x), \\
 B(x)&=B(x * B(x)).\\
\end{cases}
\]
Hence, $B (B(x) * x) = B (x * B(x))$. If the second system holds, we get the same equality.
\end{proof}

Taking  in attention Example \ref{ARO} we get than

\begin{corollary}
For any rack algebra $\Bbbk[X]$ the identity map is a rack and algebraic Rota--Baxter operator of weight $\lambda = -1$.
\end{corollary}

Now we are considering averaging operators.

\begin{proposition}
Any averaging operator on a rack is extended by linearity to an averaging operator on the  rack algebra.
\end{proposition}

\begin{proof}
Consider only  right averaging operators, proof for left averaging operators is similar.

Let $B : X \to X$ be  an averaging operator on the rack $(X, *)$, i.e.
$$
B(x) *  B(y) = B (x * B(y)),\qquad\qquad (x, y \in X).
$$
For any $u = \sum_{x \in X} \alpha_x \, x$ and $v = \sum_{y \in X} \beta_y \, y\in \Bbbk[X]$ where $\alpha_x, \beta_y \in \Bbbk$, we have that
\[
\begin{split}
B(u) *  B(v) &= \sum_{x, y \in X} \alpha_x \beta_y \, \left( B(x) * B(y)\right);\\
B (u * B(v)) &= B\left( \sum_{x, y \in X} \alpha_x \beta_y \, \left( x * B(y)\right) \right) = \sum_{x, y \in X} \alpha_x \beta_y \, B\left( x * B(y)\right).
\end{split}
\]
Since $B$ is an averaging operator on $X$,  we conclude that $B(u) *  B(v) = B (u * B(v))$,
so $B$ induced an averaging operator on $\Bbbk[X]$.
\end{proof}

\begin{question}
Let  $\Bbbk[\R_n]$ be a   quandle algebra of a dihedral quandle $\R_n$ over a field $\Bbbk$. Find:

-- Averaging operators on $\R_n$ and on $\Bbbk[\R_n]$;

-- Rota--Baxter operators on $\R_n$;

-- (monomial) Rota--Baxter operators on $\Bbbk[\R_n]$.
\end{question}



\section{Some generalization of conjugation quandle and core quandle} \label{GenRack}

In Example \ref{Def:5Fvvja}, for each group $G$ we constructed  the conjugation quandle $\Conj(G)$ and the core quandle $\Core(G)$.  The following definition generalizes these constructions.

\begin{definition}\label{Def:6RRRcb}
Let $(G, B)$ be a Rota--Baxter group. Define the following  two groupoids:
\begin{itemize}
\item[(i)] {\it $B$-conjugation groupoid:} $(\Conj_B(G), *_B)$ in which  $x *_B y:= B(y) x B(y)^{-1}$, for  $x, y \in G$;

\item[(ii)] {\it $B$-core  groupoid:}  $(\Core_B(G), *_B)$ in which $x *_B y:= B(y) x^{-1} B(y)$ for  $x, y \in G$.
\end{itemize}
\end{definition}

Evidently,   $(a *_B b) *_B b = a$ for all $a, b \in G$ in $\Core_B(G)$.

\begin{example}
Let $(G, B)$ be an RB-group. Let  $(\Conj_B(G), *_B)$ and $(\Core_B(G), *_B)$ be groupoids defined in Definition \ref{Def:6RRRcb}. Let $B_{[-1]}$ and $B_{[0]}$ be maps from Example \ref{Exccgaf}.
\begin{itemize}
\item[(i)] If $B:=B_{[-1]}$, then the operation $*_B$ in the groupoid $\Conj_B(G)$ has the form $x *_B y:= y^{-1} x y$ and $\Conj_B(G) = \Conj(G)$ is the conjugation quandle.

\item[(ii)] If $B:=B_{[-1]}$, then the operation $*_B$ in $\Core_B(G)$ has the form $x *_B y:= y^{-1} x^{-1} y^{-1}$. Clearly,  in general,  $(\Core_B(G), *_B)$   is not a rack.

\item[(iii)] If $B:=B_{[0]}$  then the operation in $\Conj_B(G)$ has the form $x *_B y:=  x$ that means that $\Conj_B(G)$ is a trivial quandle.

\item[(iv)]  If $B:=B_{[0]}$  then the operation in $\Core_B(G)$ has the form $x *_B y:=  x^{-1}$ and  $(\Core_B(G), *_B)$ is a rack.   If $G$ is an elementary abelian $2$-group, then $\Core_B(G)$ is a quandle.
\end{itemize}
\end{example}

Let $\mathcal{RB}(G)$ be a  set of all $RB$-operators on a group $G$. If  $\mathcal{G}$ is the set of all groupoids  on the set $G$, then there exist the following  two maps:
\[
\Psi_1, \Psi_2: G \times  \mathcal{RB}(G) \to \mathcal{G},
\]
which are defined (see Definition \ref{Def:6RRRcb}) by the  rules:
\begin{equation}\label{ZZAQAZ}
\begin{split}
\Psi_1 ((G, B)) &\mapsto (G, x *_B y := B(y) x B(y)^{-1}),\\
\Psi_2 ((G, B)) &\mapsto (G, x *_B y := B(y) x^{-1} B(y)),\qquad\quad (x, y \in G).
\end{split}
\end{equation}

In Proposition \ref{Prop:2} for each RB-group $(G, B)$ was defined a new operation  $\circ$ on $G$, such that $(G, B,\circ)$ becomes a group.
Conditions on an RB-group  $(G,B)$   under which  groupoids  $(G, B, *_B)$  from  Definition \ref{Def:6RRRcb} become  racks or/and quandles, provide the following.

\begin{theorem} \label{Constr}
Let $(G, B)$ be a  Rota--Baxter group. Let  $\zeta(G)$ be the center of $G$. Let $(\Conj_B(G), *_B)$ and  $(\Core_B(G), *_B)$ be groupoids from  Definition \ref{Def:6RRRcb}.  The following conditions hold:
\begin{itemize}
\item[(i)] if $B(b)^{-1} b\in \zeta(G)$ for any $b\in G$, then  the  $B$-conjugation  groupoid $(\Conj_B(G),*_B)$ is a rack.
Moreover, if  $B(a) \in C_G(a)$ for any $a \in G$, then the  $B$-conjugation groupoid $(\Conj_B(G),*_B)$ is a quandle;

\item[(ii)] if
\[
B(c) B(b)^{-1} B(c) = B \left( B(c)^{-1} B(b)^{-1} B(c)^{-1} \right),   \qquad\qquad  (b, c \in G)
\]
then the $B$-core groupoid $(\Core_B(G),*_B)$ is a rack.
Moreover, if $\left( a B(a) \right)^2 = 1$ for any $a \in G$, then $(\Core_B(G),*_B)$ is a quandle.
\end{itemize}
\end{theorem}

\begin{proof} (i) For each $a,b \in G$ there is a unique $x = B(a)^{-1} b B(a) \in G$ such that $x *_B  a = b$. Since
\[
\begin{split}
(a*_Bb)*_Bc &= B(c) B(b) a B(b)^{-1} B(c)^{-1};\\
(a*_Bc)*_B(b*_Bc)& = B(c) b a b^{-1} B(c)^{-1},
\end{split}
\]
we obtain that $B(b) a B(b)^{-1} = b a b^{-1}$ by  (Q3). This  yields that
  $B(b)^{-1} b\in \zeta(G)$. According to (Q2), we have that  $B(a) a B(a)^{-1} = a$, so    $B(a)\in C_G(a)$.

(ii) The proof is similar. \end{proof}

Let $I$ be a set. Let $(G, B_i)$, $i \in I$ be RB-groups, such that  the set of their images
\begin{equation}\label{ZNNFTbZ}
\{ \Psi_k (G, B_i), \mid   i\in I,  k= 1,2; \quad \Psi_1,\Psi_2 \quad \text{from \eqref{ZZAQAZ}}\}
\end{equation}
consists of  quandles which were  constructed in   Theorem \ref{Constr}.
Let   $(G, *_i, i \in I )$ be the algebraic system with the  set of quandle operations $*_i: = *_{B_i}$. This algebraic system is called an {\it  $I$-quandle system}.

Define a product of quandle operations $*_i$ (see \cite{BF}) by the rule:
\begin{equation}\label{yyseVVq}
g *_s *_t h :=  (g *_s h) *_t h,  \qquad\qquad  (g, h \in G).
\end{equation}

According to \cite{BF}, a pair  $(G, *_s *_t)$  in which $s,t\in I$ is a quandle if and only if
\begin{equation}\label{BBBvB}
(g *_s h) *_t q = (g *_t q) *_s (h *_t q),  \qquad\qquad  (g, h \in G, s, t \in I).
\end{equation}
An $I$-quandle system $(G, *_i, i \in I)$  in  which  \eqref{BBBvB} holds for any $s,t\in I$
is  called a {\it multi-quandle} (see Turaev \cite{T}).
Turaev  used multi-quandles for construction of some knot invariants. A similar conception of  multi-groups were introduced and studied  in \cite{BNY-1, Ko}.

\begin{proposition}
 Let $I$ be a set. Let $(G, B_i)$, $i\in I$ be RB-groups such that   $(G, *_i, i \in I)$ is an $I$-quandle system, in which
\begin{equation}\label{HHffdpl}
x *_i y: (= x *_{B_i} y )= B_i(y) x B_i(y)^{-1},\qquad  (x, y \in G, i \in I).
\end{equation}
This quandle system   is a multi-quandle if and only if
\[
B_t(q)^{-1} B_s\left( B_t(q) h B_t(q)^{-1} \right)^{-1} B_t(q)B_s(h)
\]
lies in the center $\zeta(G)$.
\end{proposition}
\begin{proof}
Let $*_i$ be a   quandle operation on $G$ for each $i$ defined by \eqref{HHffdpl}.
Clearly,
\[
\begin{split}
(g *_s h) *_t q &= g^{B_s(h)^{-1} B_t(q)^{-1}};\\
(g *_t q) *_s (h *_t q)&=g^{B_t(q)^{-1} B_s\left(B_t(q) h B_t(q)^{-1}\right)^{-1}}.
\end{split}
\]
According to \eqref{BBBvB}, we obtain that $
g^{B_s(h)^{-1} B_t(q)^{-1}} = g^{B_t(q)^{-1} B_s\left(B_t(q) h B_t(q)^{-1}\right)^{-1}} $ and
\[
 B_t(q)^{-1} B_s\left( B_t(q) h B_t(q)^{-1} \right)^{-1} B_t(q)B_s(h)\in\zeta(G).
\]
\end{proof}

\begin{proposition}
Let $I$ be a set. Let $(G, B_i)$, $i\in I$ be RB-groups such that   $(G, *_i, i \in I)$ is an $I$-quandle system, in which
\begin{equation}\label{HVVfdl}
x *_i y:( = x *_{B_i} y  )= B_i(y) x^{-1} B_i(y),\qquad  (x, y \in G, i \in I).
\end{equation}
This quandle system  is a multi-quandle if and only if for any $h, q \in G$ and $s, t \in I$ holds
\[
\begin{split}
B_t(q) B_s(h)^{-1} g B_s(h)^{-1} B_t(q) =& B_s \left( B_t(q) h^{-1} B_t(q) \right) \times\\
 &\times B_t(q)^{-1} g B_t(q)^{-1} B_s \left( B_t(q) h^{-1} B_t(q) \right).
\end{split}
\]
\end{proposition}

\begin{proof}
Let $*_i$ be a   quandle operation on $G$ for each $i$ defined by \eqref{HVVfdl}.
Clearly,
\[
\begin{split}
(g *_s h) *_t q &= B_t(q) B_s(h)^{-1} g B_s(h)^{-1} B_t(q);\\
 (g *_t q) *_s (h *_t q) &= B_s \left( B_t(q) h^{-1} B_t(q) \right)  B_t(q)^{-1} g B_t(q)^{-1} B_s \left( B_t(q) h^{-1} B_t(q) \right).
\end{split}
\]
According to \eqref{BBBvB}, we obtain that
\[
\begin{split}
B_t(q) B_s(h)^{-1} g B_s(h)^{-1} B_t(q) &= B_s \left( B_t(q) h^{-1} B_t(q) \right) \times \\
&\times B_t(q)^{-1} g B_t(q)^{-1} B_s \left( B_t(q) h^{-1} B_t(q) \right).
\end{split}
\]
\end{proof}


\section{Open problems}

\begin{question}
We know definition of Rota--Baxter operator on Lie algebra. Is it possible to define Rota--Baxter operators on a  Lie super-algebra?
\end{question}

\subsection{Other operators}
On arbitrary   algebra $A$ it is possible defined the following operators:
\begin{itemize}
\item[(i)] Reynolds operator: $R(x) R(y)  = R(x R(y) + R(x)y - R(x)R(y))$;

\item[(ii)] Nijenhuis operator: $R(x) R(y)  = R(x R(y) + R(x) y - R(x y))$;

\item[(iii-a)] Right averaging operator: $R(x) R(y)  = R(x R(y))$;

\item[(iii-b)] Left averaging operator: $R(x) R(y)  = R(R(x) y)$;

\item[(vi)] Differential operator of weight $\lambda$:  $d(x  y) = d(x) y + x d(y) + \lambda d(x)d(y)$.
\end{itemize}

\begin{question}
Is it possible to define these operators on groups (quandles, racks)?

\end{question}

\begin{question}
Let $(G, B)$ be a Rota--Baxter group. Under which conditions  $B$ defines a $RB$-operator on the conjugacy  quandle $\Conj(G)$, on the core quandle $\Core(G)$,  on the generalized Alexander quandle $Alex(G, \varphi)$, respectively?
\end{question}

\subsection{From Lie racks to Leibniz algebra}

 In \cite{GLS} it has been proven that if $(G, B)$ is a Rota--Baxter Lie group, then the tangent map $B$ at identity is a Rota--Baxter operator of weight $1$ on the Lie algebra of the Lie group $G$.

A {\it pointed rack}  $(X, *, 1)$ is a rack $(X,*)$, with a distinguished element $1 \in X$ such that
\[
x * 1 = x, \quad\text{and}\quad  1 * x = 1 \qquad (x \in X).
\]
 Imitating the notion of a Lie group, the smooth version of a pointed rack is called {\it Lie rack}.
Let $(X, *, 1)$ be a smooth pointed rack and  let $T_1X$ be the  tangent space of it in~$1$.
In \cite{Kin} was proved that there exists
a bilinear mapping $[\cdot , \cdot ] :  T_1X \times T_1X \to T_1X$ such that
 $(T_1X, [\cdot , \cdot ])$ is a Leibniz algebra.

\begin{question}
Let $(X, *, 1)$ be a  smooth pointed rack   with a  relative Rota--Baxter or averaging operator $B$, and  $(T_1X, [\cdot , \cdot ])$ be the corresponding  Leibniz algebra. What operator induces $B$ on this algebra?
\end{question}


\end{document}